\documentclass[reqno,12pt]{amsart}

\usepackage{amsmath}
\usepackage{amsfonts}
\usepackage{amssymb}
\usepackage{eufrak}
\usepackage{amscd}
\usepackage{amsthm}
\usepackage{amstext}
\usepackage[all]{xy}

\newcommand{\Ker}{\operatorname{Ker}}

   \theoremstyle{plain}
   \newtheorem{thm}{Theorem}[section]
   \newtheorem{prop}[thm]{Proposition}
   \newtheorem{lem}[thm]{Lemma}
   \newtheorem{cor}[thm]{Corollary}
   \theoremstyle{definition}

   \theoremstyle{remark}

	\newtheorem{problem}[thm]{Problem}

\author{V. Manuilov}

\date{}

\address{Moscow State University,
Leninskie Gory 1, Moscow, 
119991, Russia}

\email{manuilov@mech.math.msu.su}

\title{On the $C^*$-algebra of matrix-finite bounded operators}

\begin{document}

\maketitle

\begin{abstract}
Let $H$ be a separable Hilbert space with a fixed orthonormal basis. Let $\mathbb B^{(k)}(H)$ denote the set of operators, whose matrices have no more than $k$ non-zero entries in each line and in each column. The closure of the union (over $k\in\mathbb N$) of $\mathbb B^{(k)}(H)$ is a $C^*$-algebra. We study some properties of this $C^*$-algebra. We show that this $C^*$-algebra is not an $AW^*$-algebra, has a proper closed ideal greater than compact operators, and its group of invertibles is contractible. 

\end{abstract}

\section*{Introduction}

Hilbert spaces often come out with a fixed special orthonormal basis, e.g. $l^2(X)$ for a countable discrete space $X$, and the correspondence between operators and their matrices makes more sense. We got interested in the operators, whose matrices with respect to the fixed basis satisfy the following property: there is some $k\in\mathbb N$ such that each line and each column of the matrix has no more than $k$ non-zero entries. The norm closure of the union (over $k\in\mathbb N$) of these operators turns out to be a $C^*$-algebra, denoted by $\mathbb B_f(H)$. We establish some properties of this $C^*$-algebra, which is, in some aspects, similar to the algebra $\mathbb B(H)$ of all bounded operators on the Hilbert space $H$, e.g. the Kuiper theorem on contractibility of the group of invertibles holds also for invertibles in $\mathbb B_f(H)$. On the other hand, $\mathbb B_f(H)$ is not a von Neumann algebra, and even not an $AW^*$-algebra, and the quotient $\mathbb B_f(H)/\mathbb K(H)$ is not simple.

\section{Matrix-finite bounded operators on a Hilbert space}

Let $H$ be a separable Hilbert space with a fixed orthonormal basis $\{e_n\}_{n\in\mathbb N}$. For $k\in\mathbb N$, denote by $\mathbb B^{(k)}_L(H)$ (resp. $\mathbb B^{(k)}_C(H)$) the set of all bounded operators on $H$ such that each line (resp. each column) of their matrix (with respect to the fixed basis) contains no more than $k$ non-zero entries. Note that 
$$
\mathbb B^{(k)}_L(H)\subset\mathbb B^{(l)}_L(H)\mbox{\ when\ } k<l, \mbox{\ and\ } \mathbb B^{(k)}_C(H)=(\mathbb B^{(k)}_L(H))^*. 
$$
Set also 
$$
\mathbb B^{(k)}(H)=\mathbb B^{(k)}_L(H)\cap\mathbb B^{(k)}_C(H).
$$  

Let $a,b\in\mathbb B(H)$, $A=(a_{ij})$, $B=(b_{ij})$ their matrices.

\begin{lem}\label{sum}
If $a,b\in\mathbb B^{(k)}_L(H)$ then $a+b\in\mathbb B^{(2k)}_L(H)$.

\end{lem} 
\begin{proof}
Obvious.

\end{proof}

\begin{lem}\label{product}
If $a,b\in\mathbb B^{(k)}_L(H)$ then $ab\in\mathbb B^{(k^2)}_L(H)$.

\end{lem} 
\begin{proof}
Let $c_{il}=\sum_{j\in\mathbb N}a_{ij}b_{jl}$. Fix $i$. There exist $j_1,\ldots,j_k\in\mathbb N$ such that $a_{ij}=0$ if $j\notin\{j_1,\ldots,j_k\}$. For each $j_m$, $m=1,\ldots,k$, there exist $l_1^{(m)},\ldots,l_k^{(m)}\in\mathbb N$ such that $b_{j_m l}=0$ if $l\notin\{l_1^{(m)},\ldots,l_k^{(m)}\}$. So $c_{il}=0$ for $l\notin\{l_n^{(m)}\}_{n,m=1}^k$, hence the $i$-th line contains no more than $k^2$ non-zero entries.  

\end{proof}

Set $\mathbb B^{(\infty)}(H)=\cup_{k\in\mathbb N}\mathbb B^{(k)}(H)$. This algebra was defined and studied in \cite{Cortinas}, {Abadie-Cortinas}.

\begin{cor}
$\mathbb B^{(\infty)}(H)$ is a unital $*$-subalgebra in $\mathbb B(H)$.

\end{cor}

Note that the correspondence between operators and their matrices is more clear for operators in $\mathbb B^{(\infty)}(H)$.

\begin{lem}\label{ulf-matrix}
Let $(a_{ij})_{i,j\in\mathbb N}$ be a matrix such that each line and each column contains no more than $k$ non-zero entries, and there exists $C$ such that $|a_{ij}|<C$ for any $i,j\in\mathbb N$. Then the operator $a$ given by this matrix is bounded. 

\end{lem}
\begin{proof}
Let the non-zero elements in the $i$-th line are $a_{i,j_1^{(i)}},\ldots,a_{i,j_k^{(i)}}$. Define matrices $a^{(1)},\ldots,a^{(k)}$ by 
$$
(a^{(m)})_{ik}=\left\lbrace\begin{array}{cl}a_{i,j_m^{(i)}}&\mbox{\ if\ }k=j_m^{(i)}\\0&\mbox{\ if\ }k\neq j_m^{(i)}.\end{array}\right.
$$
Then each matrix $a^{(m)}$ has a single non-zero element in each line, and $a=a^{(1)}+\cdots+a^{(k)}$. 

Let $\xi=(\xi_n)_{n\in\mathbb N}\in H$. Then $a^{(m)}\xi=(a_{1,j_m^{(1)}}\xi_{j_m^{(1)}},a_{2,j_m^{(2)}}\xi_{j_m^{(2)}},\ldots)$, hence
$$
\|a^{(m)}\xi\|^2=|a_{1,j_m^{(1)}}\xi_{j_m^{(1)}}|^2+|a_{2,j_m^{(2)}}\xi_{j_m^{(2)}}|^2+\cdots <C^2(|\xi_{j_m^{(1)}}|^2+|\xi_{j_m^{(2)}}|^2+\cdots)\leq C^2k\sum_{j\in\mathbb N}|\xi_j|^2,
$$
as each $j$ appears not more than $k$ times among the numbers $j_m^{(1)},j_m^{(2)},\ldots$. 

Thus $\|a^{(m)}\|^2<C^2k$. 
As $\|a\|\leq\|a^{(1)}\|+\cdots+\|a^{(k)}\|$, we get $\|a\|<Ck^{3/2}$.

\end{proof}

Denote by $\mathbb B_f(H)$ the norm closure of $\mathbb B^{(\infty)}(H)$. Then it is a $C^*$-algebra.

\begin{lem}\label{matrix}
$M_m(\mathbb B_f(H))\cong\mathbb B_f(H)$ for any $m\in\mathbb N$.

\end{lem}
\begin{proof}
Obvious.

\end{proof}

Let $\mathbb K(H)$ denote the $C^*$-algebra of compact operators on $H$.

\begin{lem}\label{compact}
$\mathbb K(H)\subset\mathbb B_f(H)$.

\end{lem}
\begin{proof}
Any compact operator can be approximated by operators with only finite number of non-zero entries.

\end{proof}

\section{Difference from $\mathbb B(H)$}

The following examples show that the $C^*$-algebra $\mathbb B_f(H)$ is smaller than $\mathbb B(H)$, but not too much smaller.

\begin{prop}\label{isometry}
Define the vectors $a_n\in H$, $n\in\mathbb N$, by 
$$a_1=(1,0,0,\ldots);\ a_2=(0,\frac{1}{\sqrt{2}},\frac{1}{\sqrt{2}},0,0,\ldots);\ a_3=(0,0,0,\frac{1}{\sqrt{3}},\frac{1}{\sqrt{3}},\frac{1}{\sqrt{3}},0,0,\ldots);\ \ldots
$$ 
Define an isometry $v$ by $v(e_n)=a_n$. Then $v\notin\mathbb B_f(H)$. 

\end{prop}
\begin{proof}
If $v\in\mathbb B_f(H)$ then for any $\varepsilon>0$ there exists $k\in\mathbb N$ and $b\in\mathbb B_f^{(k)}(H)$ such that $\|v-b\|<\varepsilon$. In particular, $\|(v-b)e_n\|<\varepsilon$ for any $n\in\mathbb N$. The vector $be_n$ has no more than $k$ non-zero coordinates, hence $\|a_n-be_n\|^2\geq \frac{n-k}{n}$. Taking $\varepsilon<\frac{1}{2}$ and $n>2k$, we get a contradiction.

\end{proof}

The same argument can be used in the next example:

\begin{prop}\label{unitary}
Let $L\subset H$ be the closed subspace spanned by $a_1,a_2,\ldots$, and let $b_1,b_2,\ldots$ be an orthonormal basis for $L^\perp$. Set $u(e_{2n-1})=a_n$, $u(e_{2n})=b_n$. Then $u$ is a unitary, and $u\notin\mathbb B_f(H)$.

\end{prop}

\begin{prop}\label{projection}
Let $H=\oplus_{n\in\mathbb N}\mathbb C^n$ with the basis consisting of the standard bases of $\mathbb C^n$, $n\in\mathbb N$. Let $p_n=\left(\begin{smallmatrix}\frac{1}{n}&\cdots&\frac{1}{n}\\\vdots&\ddots&\vdots\\\frac{1}{n}&\cdots&\frac{1}{n}\end{smallmatrix}\right)$ denote the projection onto the line $(x,\ldots,x)$, $x\in\mathbb C$, in $\mathbb C^n$. Set $p=\oplus_{n\in\mathbb N}p_n$. Then $p\in\mathbb B_f(H)$. Note that we have $p=vv^*$, where $v\in\mathbb B(H)\setminus\mathbb B_f(H)$ is the isometry from Proposition \ref{isometry}.

\end{prop}

\begin{proof}
The proof uses expander graphs \cite{Lubotzky} and essentially is contained in \cite{Higson-Lafforgue-Skandalis}, where it is shown that the direct sum of one-dimensional projections $\oplus_{n\in\mathbb N}p_{m_n}$ lies in the uniform Roe algebra (hence in $\mathbb B_f(H)$, cf. Lemma \ref{Roe} below) for certain increasing sequences $\{m_n\}$ of sizes. We only add a remark that one may take $m_n=n$. 

Let $X=(V,E)$ be a $d$-regular graph, without loops and multiple edges, with the set $V$ of vertices and the set $E$ of edges. Let $|V|=m$. The $m$-dimensional Hilbert space $l^2(V)$ is endowed with the standard basis consisting of characteristic functions of the vertices. The Laplacian on $l^2(V)$ is the positive operator $\mathcal L$ given by the matrix $(L_{v,w})_{v,w\in V}$, where 
$$
L_{v,w}=\left\lbrace\begin{array}{rl}1,&\mbox{if\ }v=w;\\-\frac{1}{d},&\mbox{if\ }v,w\mbox{\ are\ adjacent;}\\0,&\mbox{otherwise}.\end{array}\right.
$$ 
It is known that $\|\mathcal L\|\leq 2$ for any $m\in\mathbb N$. It is also easy to see that the matrix $L$ has no more than $d+1$ non-zero entries in each line and in each column for any $m\in\mathbb N$.

The smallest eigenvalue $\lambda_0$ of $\mathcal L$ is zero, and the corresponding eigenfunctions are constants, so the projection onto the kernel of $\mathcal L$ is the spectral projection onto this eigenspace, and its matrix equals $p_m$.

If $d\geq 5$ then for each $n\in\mathbb N$ there exists a $\frac{1}{2}$-expander graph $X_n=(V_n,E_n)$ with $|V_n|=2n$ (\cite{Lubotzky}, Proposition 1.2.1), and it follows (\cite{Lubotzky}, Section 4.2) that the second smallest eigenvalue $\lambda_1$ of each Laplacian $\mathcal L_n$ of $X_n$ satisfies $\lambda_1\geq\delta$ for some $\delta>0$. 

Let $\{f_s\}_{s\in\mathbb N}$ be a sequence of polynomials such that 
\begin{itemize}
\item
$f_s(0)=1$ for any $s\in\mathbb N$;
\item
$|f_s(t)|\leq\frac{1}{s}$ for any $t\in[\delta,2]$.
\end{itemize} 

Set $H=\oplus_{n\in\mathbb N}l^2(V_n)$ with the basis obtained by uniting the bases of each $l^2(V_n)$, and let $\mathcal L=\oplus_{n\in\mathbb N}\mathcal L_n$.

Note that each $\mathcal L_n$ has no more than $d+1$ non-zero entries in each line and in each column. By Lemma \ref{sum} and Lemma \ref{product}, we conclude that for any $s\in\mathbb N$ there exists $k=k(s)\in\mathbb N$ such that each line and each column of $f_s(\mathcal L_n)$ contains no more than $k$ non-zero entries.
Therefore, $f_s(\mathcal L)=\oplus_{n\in\mathbb N}\mathcal L_n\in\mathbb B^{(k)}(H)$. As the spectrum of $\mathcal L$ lies in $\{0\}\cup[\delta,2]$, the operators $f_s(\mathcal L)$ converge to the spectral projection $p'=\oplus_{n\in\mathbb N}p_{2n}$ onto the kernel of $\mathcal L$: 
$$
\lim_{s\to\infty}\|f_s(\mathcal L)-p'\|=0,
$$ 
hence $p'\in\mathbb B_f(H)$. 

Similarly, $p''=p_1\oplus(\oplus_{n\in\mathbb N}(p_{2n}\oplus p_{2n}\oplus 0))\in\mathbb B_f(H)$. Note that 
$$
\left\|p_{2n+1}-\left(\begin{matrix}p_{2n}&0\\0&0\end{matrix}\right)\right\|=\left\|\frac{1}{2n+1}\left(\begin{matrix}-p_{2n}&l\\ l^* &1\end{matrix}\right)\right\|\leq \frac{1}{2n+1}(\|p_{2n}\|+2\|l\|+1)=\frac{2+2\sqrt{2n}}{2n+1},
$$
where $l$ is the column of $2n$ units. So, $p-p''\in\mathbb K(H)$, and, by Lemma \ref{compact}, $p\in\mathbb B_f(H)$.

\end{proof}

Note that, by spectral theorem, any selfadjoint $a\in\mathbb B(H)$ can be approximated by linear combinations of projections, so if all projections in $H$ would lie in $\mathbb B_f(H)$ then $a\in\mathbb B_f(H)$, hence any $b\in\mathbb B(H)$ must lie in $\mathbb B_f(H)$, which is false. Therefore, there exists a projection in $\mathbb B(H)\setminus\mathbb B_f(H)$. Here is an explicit example.

\begin{prop}
Let $u$ be the unitary from Proposition \ref{unitary} on a Hilbert space $H$, and let $p=\frac{1}{2}\left(\begin{matrix}1&u\\u^*&1\end{matrix}\right)$ be a projection in $H\oplus H$ (with the basis obtained by uniting bases of the two copies of $H$). Then
$p\notin\mathbb B_f(H\oplus H)$.

\end{prop}
\begin{proof}
This follows from Lemma \ref{matrix} and from $u\notin\mathbb B_f(H)$.

\end{proof}

\section{Ideals}

It is known that $\mathbb K(H)$ is the unique closed two-sided ideal in $\mathbb B(H)$. This is not true for $\mathbb B_f(H)$.

Let $a\in\mathbb B_f(H)$, $(a_{ij})_{i,j\in\mathbb N}$ the matrix of $a$ (recall that the basis is fixed). Set 
$$
\mathbb I(H)=\{a\in\mathbb B_f(H):\lim_{n\to\infty}\sup_{i,j\geq n}|a_{ij}|=0\}. 
$$
Obviously, $\mathbb K(H)\subset\mathbb I(H)$. It is easy to see that the projection $p$ from Proposition \ref{projection} lies in $\mathbb I(H)$ and is not compact, hence $\mathbb I(H)$ is strictly bigger than $\mathbb K(H)$.

A similar ideal in uniform Roe algebras was studied, e.g. in \cite{Roe-Willett}, under the name of the ideal of ghost operators.

\begin{thm}
$\mathbb I(H)$ is the maximal closed two-sided proper ideal in $\mathbb B_f(H)$.

\end{thm}
\begin{proof}
First, note that $\mathbb I(H)$ is closed. Indeed, suppose that $(a^{(\iota)})_{\iota\in I}$ converges to $a\in\mathbb B_f(H)$, and $a^{(\iota)}\in\mathbb I(H)$ for any $\iota\in I$. Fix $\varepsilon>0$, then there is $\kappa\in I$ such that $\|a^{(\iota)}-a\|<\varepsilon$ for any $\iota\geq\kappa$. Therefore, $|a^{(\iota)}_{ij}-a_{ij}|<\varepsilon$ for any $i,j\in\mathbb N$. Then 
$$
|a_{ij}|\leq|a^{(\kappa)}_{ij}|+|a_{ij}-a^{(\kappa)}_{ij}|<|a^{(\kappa)}_{ij}|+\varepsilon. 
$$
As $a^{(\kappa)}\in\mathbb I(H)$, there exists $N$ such that $|a^{(\kappa)}_{ij}|<\varepsilon$ for any $i,j>N$. Then $|a_{ij}|<2\varepsilon$ for any $i,j>N$.

Now let $a\in\mathbb I(H)$, $b\in\mathbb B^{(k)}(H)$ for some $k$. Let us check that $ab\in\mathbb I(H)$. Due to the involution, this would imply $ba\in\mathbb I(H)$. Set $c_{il}=\sum_{j\in\mathbb N}a_{ij}b_{jl}$. Take $\varepsilon>0$ and find $N$ such that $|a_{ij}|<\varepsilon$ for $i,j>N$. After fixing $N$, we can find $M$ such that $b_{jl}=0$ for any $j\leq N$ and any $l>M$ (see Lemma \ref{product}). Then, for $i>N$ and $l>M$ we have
$$
|c_{il}|=\Bigl|\sum_{j\in\mathbb N}a_{ij}b_{jl}\Bigr|=\Bigl|\sum_{j>N}a_{ij}b_{jl}\Bigr|<\varepsilon k\|b\|
$$ 
(we use here that the sum contains not more than $k$ non-zero summands and that $|b_{jl}|\leq\|b\|$ for all $j,l\in\mathbb N$).
Thus $\mathbb I(H)$ is a two-sided ideal in $\mathbb B_f(H)$. 

It remains to show that $\mathbb I(H)$ is the maximal proper closed two-sided ideal in $\mathbb B_f(H)$. Assume that there exists an ideal $\mathbb J$ and a selfadjoint $a\in\mathbb J\setminus\mathbb I(H)$. We are going to construct a subspace $L\subset H$ and an isometry $u:H\to L$ in two different cases.
First, consider the case when there exists a sequence $\{i_n\}_{n\in\mathbb N}$ such that $\|a_{i_n,i_n}\|>\delta$ for some $\delta>0$. As $\mathbb B^(\infty)(H)$ is dense in $\mathbb B_f(H)$, there exists $k\in\mathbb N$ and a selfadjoint $a^{(k)}\in\mathbb B^{(k)}(H)$ such that $\|a-a^{(k)}\|<\delta/4$. Then $|a^{(k)}_{i_n,i_n}|>3\delta/4$. Let us pass to a subsequence $\{i_{n_m}\}_{m\in\mathbb N}$ inductively: set $n_1=1$, and if $m_{n}$ is already fixed then let $n_{m+1}$ be defined by the condition
\begin{equation}\label{condition}
i_{n_{m+1}}>l \mbox{\ for\  any\ } l\in\mathbb N \mbox{\ such\  that\  at\  least\  one\  of\ } a^{(k)}_{i_{n_1},l},\ldots,a^{(k)}_{i_{n_m},l} \mbox{\ is\  non-zero}
\end{equation}
(for each $n$, there is not more than $k$ non-zero entries $a^{(k)}_{i_n,j}$).
Then let $L\subset H$ be a closed subspace spanned by the vectors $e_{i_{n_m}}$, $m\in\mathbb N$, and let $p_L\in\mathbb B_f(H)$ be the projection onto $L$. 
By (\ref{condition}), $p_La^{(k)}|_L$ is a diagonal operator: 
$$
a^{(k)}_{i_{n_m},i_{n_{m'}}}=0 \mbox{\ if\ } m\neq m',\quad\mbox{and}\quad |a^{(k)}_{i_{n_m},i_{n_m}}|>3\delta/4. 
$$
Then $\|p_La^{(k)}\xi\|>\frac{3\delta}{4}\|\xi\|$ and 
$$
\|p_La\xi\|\geq \|p_La^{(k)}\xi\|-\|p_L(a^{(k)}-a)\xi\|>\frac{3\delta}{4}\|\xi\|-\frac{\delta}{4}\|\xi\|=\frac{\delta}{2}\|\xi\|
$$
for any $\xi\in L$. Let $u:H\to L$ be the isometry defined by $u(e_m)=e_{i_{n_m}}$. Then $u\in\mathbb B_f(H)$. 

Otherwise, assume that $\lim_{i\to\infty}a_{ii}=0$. As $a\notin\mathbb I(H)$, there exists some $\delta>0$ and sequences $\{i_n\},\{j_n\}_{n\in\mathbb N}$ such that $i_{n+1},j_{n+1}>\max(i_n,j_n)$, $j_n\neq i_n$, and $|a_{i_n,j_n}|>\delta$ for any $n\in\mathbb N$. As in the first case, find a selfadjoint $a^{(k)}\in\mathbb B^{(k)}(H)$ such that $\|a-a^{(k)}\|<\delta/6$. Then in the two-by-two matrices 
\begin{equation}\label{matrix1}
\left(\begin{matrix}a^{(k)}_{i_n,i_n}&a^{(k)}_{i_n,j_n}\\a^{(k)}_{j_n,i_n}&a^{(k)}_{j_n,j_n}\end{matrix}\right)
\end{equation}
the diagonal entries satisfy $|a^{(k)}_{i,i}|<\delta/6$ for sufficiently great $n$, while the off-diagonal entries have modulus greater than $5\delta/6$, hence  there exists some $n_0\in\mathbb N$ such that the matrices (\ref{matrix1}) are bounded from below by $2\delta/3$ for $n>n_0$. 

Once again let us pass inductively to a subsequence $\{n_m\}_{m\in\mathbb N}$. Set $n_1\geq n_0$, and if $n_m$ is already fixed then let
$n_{m+1}$ be defined by the condition
$$
\begin{array}{c}
\min(i_{n_{m+1}},j_{n_{m+1}})>l\\ \mbox{\ for\  any\ } l\in\mathbb N 
\mbox{\ such\  that\  at\  least\  one\  of\ } a^{(k)}_{i_{n_1},l},\ldots,a^{(k)}_{i_{n_m},l}, a^{(k)}_{j_{n_1},l},\ldots,a^{(k)}_{j_{n_m},l}\mbox{\ is\  non-zero}.
\end{array}
$$

Let $L$ be a closed subspace spanned by $e_{i_{n_{m}}},e_{j_{n_{m}}}$, $m\in\mathbb N$. Then $p_La^{(k)}$ is a block-diagonal operator with two-dimensional blocks of the form (\ref{matrix1}), hence $\|p_La^{(k)}\xi\|>\frac{2\delta}{3}$ for any $\xi\in L$. Then
$$
\|p_La\xi\|\geq \|p_La^{(k)}\xi\|-\|p_L(a^{(k)}-a)\xi\|>\frac{2\delta}{3}\|\xi\|-\frac{\delta}{6}\|\xi\|=\frac{\delta}{2}\|\xi\|
$$
for any $\xi\in L$. Similarly to the first case, let $u:H\to L$ be the isometry defined by $u(e_{2m-1})=e_{i_{n_m}}$, $u(e_{2m})=e_{j_{n_m}}$. Again we have $u\in\mathbb B_f(H)$. 

The further argument is standard \cite{Calkin}. In both cases, for any $\xi\in H$ we have $\|u^*p_Lau\xi\|>\frac{\delta}{2}\|\xi\|$ for any $\xi\in H$, hence $u^*p_Lau$ is invertible, therefore, $1\in\mathbb J$ and $\mathbb J=\mathbb B_f(H)$.

\end{proof}

\section{Similar $C^*$-algebras}

We know three $C^*$-algebras, which seem similar to $\mathbb B_f(H)$. We don't know if they all are the same or different, but it is easy to see that $\mathbb B_f(H)$ lies in all of them.

In \cite{Tanbay} the following $C^*$-subalgebra of $\mathbb B(H)$ was defined. Let $a\in\mathbb B(H)$, $(a_{ij})_{i,j\in\mathbb N}$ its matrix. The operator $a$ is called $l_1$-bounded if there is $m\in(0,\infty)$ such that $\sum_{i=1}^\infty|a_{ij}|<m$ for any $j\in\mathbb N$ and $\sum_{j=1}^\infty|a_{ij}|<m$ for any $i\in\mathbb N$. It is shown in \cite{Tanbay} that $l_1$-bounded operators form a $*$-algebra, and its norm-closure is a $C^*$-algebra. Let us denote it by $\mathbb B_T(H)$.

\begin{lem}
$\mathbb B_F(H)\subset\mathbb B_T(H)$.

\end{lem}
\begin{proof}
Let $a\in\mathbb B^{(k)}(H)$. Then $\sum_{j=1}^\infty|a_{ij}|\leq k\|a\|$.

\end{proof} 

Another $C^*$-subalgebra of $\mathbb B(H)$ can be defined as follows.

Let $\omega$ be a free ultrafilter on $\mathbb N$. Set 
$$
\mathcal H=\{\xi=(\xi_n)_{n\in\mathbb N}:\sup_n\|\xi_n\|<\infty\},
$$
$$
\langle\xi,\eta\rangle=\lim_\omega \langle\xi_n,\eta_n\rangle,\quad \xi,\eta\in\mathcal H,
$$ 
$$
\mathcal H_0=\{\xi\in\mathcal H:\langle\xi,\xi\rangle=0\}, \quad H^\omega=\mathcal H/\mathcal H_0. 
$$
Then $H^\omega$ is a (non-separable) Hilbert space, and $\mathbb B(H)$ is obviously represented on $H^\omega$ by 
$$
\pi(a)\xi=(a\xi_n)_{n\in\mathbb N}, \quad \pi:\mathbb B(H)\to\mathbb B(H^\omega). 
$$
As $H^\omega$ contains a copy of $H$ consisting of constant sequences, $\pi$ is injective.

For $k\in\mathbb N$ let $\mathcal H^{(k)}$ denote the set of all sequences $\xi\in\mathcal H$ such that not more than $k$ entries are non-zero. Although each $\mathcal H^{(k)}$ is not a linear subspace, their union 
$$
\mathcal H^{(\infty)}=\cup_{k\in\mathbb N}\mathcal H^{(k)}
$$ 
is. Then $\mathcal H^{(\infty)}/\mathcal H_0$ is a linear subspace of $H^\omega$. Denote its closure by $H^\omega_f$. It is a proper subspace of $H^\omega$. 

\begin{lem}
Let $a\in\mathbb B(H)$. If $a\in\mathbb B_f(H)$ then $H^\omega_f$ is invariant for $\pi(a)$.

\end{lem}
\begin{proof}
Let $a\in\mathbb B^{(k)}(H)$, $\xi=(\xi_n)_{n\in\mathbb N}\in\mathcal H^{(l)}$. Then $a\xi_n$ has not more than $kl$ non-zero coordinates, hence $(a\xi_n)_{n\in\mathbb N}\in\mathcal H^{(kl)}$, hence $\pi(a)$ maps $\mathcal H^{(k)}+\mathcal H_0$ to $\mathcal H^{(kl)}+\mathcal H_0$, thus leaves $\mathcal H^{(\infty)}/\mathcal H_0$ invariant. 

\end{proof}

Set $\mathbb B_\omega(H)=\{a\in\mathbb B(H):\pi(a)(H^\omega_f)\subset H^\omega_f\}$. It is a $C^*$-algebra, and $\mathbb B_f(H)\subset\mathbb B_\omega(H)\subset\mathbb B(H)$.

One more $C^*$-algebra is the multiplier $C^*$-algebra $M(\mathbb I(H))$ of the ideal $\mathbb I(H)$. As $\mathbb K(H)\subset\mathbb I(H)\subset\mathbb B(H)$, we have $M(\mathbb I(H))\subset\mathbb B(H)$. 

It is easy to see that the isometry from Proposition \ref{isometry} lies in neither of the three $C*$-algebras.

\begin{problem}
Are the $C^*$-algebras $\mathbb B_f(H)$, $\mathbb B_T(H)$, $\mathbb B_\omega(H)$, $M(\mathbb I(H))$ different?

\end{problem}

\section{Embeddings}

We do not know if any separable $C^*$-algebra $A$ admits an injective $*$-homomorphism into $\mathbb B_f(H)$, but here are four  easy cases when the answer is positive. 

\begin{lem}\label{lem-action}
Let a discrete group $\Gamma$ act on a countable discrete space $X$. Then  the reduced crossed product $C_b(X)\rtimes_r\Gamma$ by the algebra $C_b(X)$ of bounded functions on $\Gamma$ admits an injective $*$-homomorphism into $\mathbb B_f(H)$. In particular, the reduced group $C^*$-algebra $C^*_r(\Gamma)$ admits an injective $*$-homomorphism into $\mathbb B_f(H)$. 

\end{lem}
\begin{proof}
Set $H=l^2(X)$ with the basis of the characteristic functions $\delta_x$ of individual points $x\in X$. Let $a\in\mathbb C[\Gamma]$, $a=\lambda_1g_1+\ldots+\lambda_rg_r$, be a finite linear combination of group elements $g_1,\ldots,g_r\in\Gamma$, $\lambda_1,\ldots,\lambda_r\in\mathbb C$. As $g\delta_x=\delta_{gx}$, $g^{-1}\delta_x=\delta_{g^{-1}x}$, the matrix of the operator on $l^2(\Gamma)$ given by $a$ has no more than $r$ non-zero entries in each line and in each column. This proves inclusion of $C^*_r(\Gamma)$ in $\mathbb B_f(H)$. Operators of multiplication by functions on $X$ are diagonal, hence trivially lie in $\mathbb B_f(H)$. The case $X=\Gamma$ proves Lemma for $C^*_r(\Gamma)$.

\end{proof}

\begin{lem}
Let $\mathbb F_n$ be a free group on $n$ generators. Then the full group $C^*$-algebra $C^*(\mathbb F_n)$ admits an injective $*$-homomorphism into $\mathbb B_f(H)$.

\end{lem}
\begin{proof}
It is shown in \cite{Lubotzky}, Proposition 2.3, that the universal norm on $\mathbb C[\mathbb F_n]$ equals the supremum of the norms of representations of $\mathbb F_n$ which factor through a finite quotient. By Lemma \ref{lem-action}, $C^*(\mathbb F_n)$ embeds into $\mathbb B_f(H)$.

\end{proof}

Let $X$ be a metric space of bounded geometry, which means that, for any $R>0$, the number of points in each ball of raduis $R$ is uniformly bounded. The uniform Roe algebra $C^*(X)$ \cite{Roe} is generated on $H=l^2(X)$ by operators $a\in\mathbb B(H)$, whose matrices satisfy $a_{xy}=0$ if $d(x,y)>R$, $R>0$ (here $x,y\in X$, $a_{xy}=\langle\delta_x,a\delta_y$). 

\begin{lem}\label{Roe}
The uniform Roe algebra $C^*(X)$ for metric spaces of bounded geometry admits an injective $*$-homomorphism into $\mathbb B_f(H)$.

\end{lem}
\begin{proof}
Obvious.

\end{proof}

Let $V=\mathbb N$, and let $G=(V,E)$ be an infinite, uniformly locally finite (ULF) graph, i.e. $E$ denotes the set of edges of $G$, two vertices $v,w\in V$ are adjacent ($v\sim w$) if there is an edge $e\in E$ that connects $v$ and $w$, each vertex $v$ has a finite number $\deg v$ of adjacent vertices, and these numbers $\deg v$ are uniformly bounded on $V$. Let $A_G=((a_G)_{v,w})_{v,w\in V}$ be the adjacency matrix, i.e. the infinite matrix defined by $(a_G)_{v,w}=\left\lbrace\begin{array}{rl}1,&\mbox{if \ }v\sim w;\\0,&\mbox{otherwise.}\end{array}\right.$ Each line and each column of the matrix $A_G$ has no more than $k$ non-zero entries for some $k\in\mathbb N$. By Lemma \ref{ulf-matrix}, this matrix determines a bounded operator $a_G$, called adjacency operator. Let $B(G)=C^*(l^\infty(V),a_G)$ be the $C^*$-subalgebra of $\mathbb B(l^2(V))$ generated by (operators of multiplication by) $l^\infty(V)$ and by $a_G$. 

\begin{lem}
The $C^*$-algebra $B(G)$ defined above admits an injective $*$-homomorphism into $\mathbb B_f(H)$.

\end{lem}
\begin{proof}
Obvious.

\end{proof}

\begin{problem}
Are there separable $C^*$-algebras that do not admit an injective $*$-homomorphism into $\mathbb B_f(H)$?

\end{problem}

\section{Contractibility of the group of invertibles}

\begin{thm}
The group $GL_f(H)$ of invertible elements of $\mathbb B_f(H)$ is contractible. 

\end{thm}
\begin{proof}
The proof follows the original proof of Kuiper \cite{Kuiper}. Let $f:\mathbb S^n\to GL_f(H)$ be a continuous map from a sphere. A preliminary step is to deform $f$ to a map $f_1$ such that $f_1(\mathbb S^n)$ is contained in the linear span of some $g_1,\ldots,g_N\in GL_f(H)$. After a small deformation, we may assume that there exists $k\in\mathbb N$ such that $g_j\in GL^{(k)}(H)$, $j=1,\ldots, N$, where $GL^{(k)}(H)=GL_f(H)\cap\mathbb B^{(k)}(H)$.

Then we have to find a sequence of unit vectors $a_i$, $a'_i$, and of finitedimensional subspaces $A_i\subset H$ with the following properties:

\begin{itemize}
\item
$A_i=\operatorname{Span}(a_i,g_1(a_i),\ldots,g_N(a_i),a'_i)$;
\item
$a'_i\perp a_i,g_1(a_i),\ldots,g_N(a_i)$, and $a'_i\perp A_l$ for $l<i$;
\item
$a_i\in \cap_{l=1}^{i-1}\left[A_l^\perp\cap\left(\cap_{j=1}^Ng_j^{-1}(A_l^\perp)\right)\right]$.
\end{itemize}

We claim that if $g_j\in GL^{(k)}(H)$ then the vectors $a_i$, $a'_i$, $i\in\mathbb N$, can be taken from the basis $\{e_n\}_{n\in\mathbb N}$ of $H$.

Indeed, one can start with $a_1=e_1$. As $g_j\in GL^{(k)}(H)$, there exists $r_1$ such that $a_1,g_1(a_1),\ldots,g_N(a_1)\in L_{r_1}$, where $L_n$ denotes the linear span of the vectors $e_1,\ldots,e_n$. Then one can take $a'_1=e_{r_1+1}$, and we have $A_1\subset L_{r_1+1}$.

Now assume that $a_1,\ldots,a_{i-1}$ and $a'_1,\ldots,a'_{i-1}$ are already fixed, and $a_1,\ldots,a_{i-1}\subset L_n$ for some $n\in\mathbb N$. Then we can find $a_i\in L_n^\perp$ such that $g_j(a_i)\in L_n^\perp$, $j=1,\ldots, N$. 
Indeed, consider $g_1(e_m)$, $m>n$. There are no more than $k$ non-zero coordinates among the first $n$ coordinates of $g_1(e_m)$ for each $m$. Then, as each line of the matrix of $g_1(e_m)$ cannot contain more than $k$ non-zero entries, the number of $m$'s (i.e. the number of columns) for which there is a non-zero coordinate among the first $n$ coordinates of $g_1(e_m)$ is finite (not greater than $kn$). The same holds for each $j=1,\ldots, N$, so there is some $m$ such that the first $n$ coordinates of $g_j(e_m)$, $j=1,\ldots,N$, are all zeroes.  

The next step of homotopy in \cite{Kuiper} is rotation of the two-dimensional subspace spanned by $g(a_i)$ and $a'_i$, where $g\in\operatorname{Span}(g_1,\ldots,g_N)$, and then another rotation in the two-dimensional subspace spanned by $a'_i$ and $a_i$. Note that, although $\mathbb B^{(k)}(H)$ is not a linear space, we have $g\in GL^{(kN)}(H)$. Both rotations involve only finite number of coordinates, hence this part of the homotopy lies in $GL^{m}(H)$ for some $m\in\mathbb N$ and connects $f_1$ with $f_2$ such that $f_2(s)a_i=|f_1(s)a_i|a_i$, $s\in\mathbb S^n$, $i\in\mathbb N$.

Let $H'=\operatorname{Span}(a_1,a_2,\ldots)$, $H_1=(H')^\perp$. The next step of homotopy connects $f_2$ with $f_3$ such that
$(f_3)|_{H_1}=(f_2)|_{H_1}$ and $f_3(s)|_{H'}=\operatorname{id}$. This homotopy only proportionally changes lengths of $f_2(s)a_i$, so does not change the number of non-zero entries. Let $p'$ and $p_1$ denote the projections onto $H'$ and $H_1$ respectively. Then $f_3(s)=p'+f_2(s)p_1$. The next step of homotopy connects $f_3(s)$ with $f_4(s)=p'+p_1f_2(s)p_1$, and does not increase the number of non-zero entries.

Recall that $H'$ is spanned by an infinite set of the vectors from the basis, so we can write it as an infinite sum of infinitedimensional subspaces $H'=H_2\oplus H_3\oplus\cdots$, each of which is spanned by an infinite set of the vectors of the basis. Then, with respect to the decomposition $H=H_1\oplus H_2\oplus H_3\oplus\cdots$, we can write $f_3(s)=\left(\begin{smallmatrix}u&&&\\&1&&\\&&1&\\&&&\ddots\end{smallmatrix}\right)$ with $u$ having in each line and in each column mot more that $C$ non-zero entries for some $C$. The standard homotopy that connects $\left(\begin{smallmatrix}u&\\&u^{-1}\end{smallmatrix}\right)$ with $\left(\begin{smallmatrix}1&\\&1\end{smallmatrix}\right)$ (applied block-diagonal-wise) increases the number of non-zero entries in each line and in each column only twice, so the last step of homotopy that connects $f_4$ with $f_5$, where $f_5(s)=1$, also lies within $GL_f(H)$.

\end{proof}

\begin{cor}
$K_0(\mathbb B_f(H))=K_1(\mathbb B_f(H))=0$. In particular, the class of the projection from Proposition \ref{projection} is trivial.

\end{cor}
\begin{proof}
This follows from Lemma \ref{matrix} and from the isomorphism $K_i(A)\cong\pi_{i\pm 1}(GL_\infty(A))$, $i\in\mathbb Z/2$, \cite{Wood}, where $A$ is a $C^*$-algebra, $GL(A)$ denotes the group of invertibles in $A$, and $GL_\infty(A)=\lim_{n\to\infty}GL(M_n(A))$.

\end{proof}

\section{Miscellanea}

\begin{lem}
The $C^*$-algebra $\mathbb B_f(H)$ has no polar decomposition, hence is not an $AW^*$-algebra.

\end{lem}
\begin{proof}
Let $v\in\mathbb B(H)$ denote the isometry from Proposition \ref{isometry}. Let $h$ be the compact operator given by $h(e_n)=\lambda_ne_n$, with $\lambda_n\neq 0$ for any $n\in\mathbb N$ and $\lim_{n\to\infty}\lambda_n=0$. Set $a=vh$, where $v$ was defined in Proposition \ref{isometry}. Compactness of $a$ implies that $a\in\mathbb B_f(H)$, and $\Ker h=0$ means that $v$ is determined in a unique way. As $v\notin\mathbb B_f(H)$, there is no polar decomposition within $\mathbb B_f(H)$. Therefore $\mathbb B_f(H)$ is not a Rickart algebra, hence not an $AW^*$-algebra \cite{Ara}.

\end{proof}

\begin{lem}
The stable topological rank of $\mathbb B_f(H)$ is $\infty$.

\end{lem}
\begin{proof}
As in $\mathbb B(H)$, the unit projection $1$ is the sum of the projections onto the subspaces spanned by odd and by even vectors of the basis, each of which is Murray--von Neumann equivalent to $1$, and both isometries $v_1$ and $v_2$ given by
$v_1(e_n)=e_{2n-1}$, $v_2(e_n)=e_{2n}$, $n\in\mathbb N$, lie in $\mathbb B_f(H)$, hence $\mathbb B_f(H)$ is properly infinite, hence has infinite stable topological rank \cite{Rieffel}.

\end{proof}

\begin{problem}
Calculate the real rank of $\mathbb B_f(H)$.

\end{problem}

\noindent
{\bf Acknowledgement.} The author is grateful to E. Troitsky for fruitful discussions.

\end{document}